\documentclass[11pt,amsmath,amsthm,amsfonts,amssymb,amscd]{article}
\usepackage{epsfig}
\usepackage[latin1]{inputenc}
\usepackage{amsthm}
\usepackage{amsmath}
\usepackage{graphicx}
\font \Bbbten=msbm10 \font \Bbbsev=msbm7 \font \Bbbfiv=msbm5
\newfam\Bbbfam \textfont\Bbbfam = \Bbbten
\scriptfont\Bbbfam=\Bbbsev \scriptscriptfont\Bbbfam=\Bbbfiv
\newcommand{\N}{\mbox{$I\!\!N$}}
\newcommand{\Z}{\mbox{$Z\!\!\!Z$}}

\newcommand{\R}{\mbox{$I\!\!R$}}

\newcommand{\Cy}{\mathcal{C}}
\newcommand{\cR}{\mathcal{R}}

\newcommand{\dist}{{\rm dist}}

\newcommand{\dif}{{\rm Diff}}
\newcommand{\diam}{{\rm diam} }
\newcommand{\trans}{\mbox{$\,{ \top} \;\!\!\!\!\!\!\raisebox{-.3ex}{$\cap$}\,$}}

\newcommand{\cita}[7]{{\sc #1, }{\it #2, }{\small #3, {\bf #4 } (#5), p.
#6-#7.}}
\newcommand{\cit}[5]{{\sc #1, }{\it #2, }{\small #3, {\bf #4 } #5.}}

\newtheorem{thm}{Theorem}[section]
\newtheorem{maintheorem}{Theorem}
\newtheorem{claim}[thm]{Claim}

\newtheorem{Df}{Definition}[section]
\newtheorem{Teo}{Theorem}[section]
\newtheorem{Lem}[Teo]{Lemma}
\newtheorem{Cor}[Teo]{Corollary}
\newtheorem{Prop}[Teo]{Proposition}
\newtheorem{Obs}[Teo]{Remark}

\title{Robust entropy expansiveness implies generic domination.}
\author{M. J. Pacifico, J. L. Vieitez
 }

\date{March, 16, 2009}

\begin{document}
\maketitle
\begin{abstract}
Let $f: M \to M$ be a $C^r$-diffeomorphism, $r\geq 1$, defined on a
compact boundaryless $d$-dimensional manifold $M$, $d\geq 2$, and let $H(p)$ be the homoclinic
class associated to the hyperbolic periodic point $p$. We prove that
if there exists a $C^1$ neighborhood $\mathcal{U}$ of $f$ such that
for every $g\in {\cal U}$ the continuation $H(p_g)$ of $H(p)$ is
entropy-expansive then there is a $Df$-invariant dominated splitting
for $H(p)$ of the form $E\oplus F_1\oplus\cdots \oplus F_c\oplus G$
where $E$ is contracting, $G$ is expanding and all $F_j$ are one
dimensional and not hyperbolic.
\end{abstract}

\section{Introduction}

In this paper we study what are the consequences at the dynamical
behavior of the tangent map $Df$ of a diffeomorphism $f:M\to M$,
assuming that $f$ is robustly entropy expansive. In this direction
we obtain that the tangent bundle has a $Df$-invariant dominated
splitting of the form $E\oplus F_1\oplus\cdots \oplus F_c\oplus G$
where $E$ is contracting, $G$ is expanding and all $F_j$ are one
dimensional and not hyperbolic.

 Let $M$ be a compact connected boundary-less Riemannian
$d$-dimensional manifold, $d\geq 2$, and $f : M \rightarrow M$ a homeomorphism.
Let $K$ be a compact invariant subset of $M$ and $\dist:M\times
M\to\R^+$ a distance in $M$ compatible with its Riemannian
structure. For $E,F \subset K$, $n\in\N$ and $\delta>0$ we say that
$E$ $(n,\delta)$-spans $F$ with respect to $f$ if for each $y\in F$
there is $x\in E$ such that $\dist(f^j(x),f^j(y))\leq \delta$ for
all $j=0,\ldots ,n-1$. Let $r_n(\delta,F)$ denote the minimum
cardinality of a set that $(n,\delta)$-spans $F$. Since $K$ is
compact $r_n(\delta,F)<\infty$. We define
$$h(f,F,\delta)\equiv \lim\sup_{n\to\infty}\frac{1}{n}\log(r_n(\delta,F))$$
and the topological entropy of $f$ restricted to $F$ as
$$h(f,F)\equiv \lim_{\delta\to 0}h(f,F,\delta)\, .$$
The last limit exists since $h(f,F,\delta)$ increases as
$\delta$ decreases to zero.

\begin{Df} \label{gammadeepsilon}
For $x\in K$ let us denote
$$\Gamma_\epsilon(x,f)\equiv \{y\in M\,/\,
d(f^n(x),f^n(y))\leq \epsilon,\, n\in\Z\}\, .$$ We will simply
write $\Gamma_\epsilon(x)$ instead of $\Gamma_\epsilon(x,f)$ when
it is understood  which $f$ we refer to.

Following Bowen (see \cite{Bo}) we say that $f/K$ is {\em
entropy-expansive} or {\em $h$-expansive} for short, if and only if
there exists $\epsilon>0$ such that
$$h^*_f(\epsilon)\equiv \sup_{x\in K}h(f,\Gamma_\epsilon(x)) = 0\, .$$
\end{Df}

 \begin{Teo}\cite[Theorem 2.4]{Bo} \label{Bowen}
 For all homeomorphism $f$ defined on a compact invariant set $K$ it holds
 $$h(f,K)\leq h(f,K,\epsilon) + h^*_f(\epsilon)\mbox{ in particular }
 h(f,K)=h(f,K,\epsilon) \mbox{ if } h^*_f(\epsilon)=0\, .$$
 \end{Teo}

A similar notion to $h$-expansiveness, albeit weaker, is the notion
of {\em asymptotically $h$-expansiveness } introduced by Misiurewicz
\cite{Mi}: let $K$ be a compact metric space and $f:K\to K$ an
homeomorphism. We say that $f$ is asymptotically $h$-expansive if
and only if
$$\lim_{\epsilon\to 0} h^*_f(\epsilon)=0\, .$$
Thus, we do not require that for a certain $\epsilon > 0$,
$h^*_f(\epsilon)=0$ but that $h^*_f(\epsilon)\to 0$ when
$\epsilon\to 0$. It has been proved by Buzzi, \cite{Bu}, that any
$C^\infty$ diffeomorphism defined on a compact manifold is
asymptotically $h$-expansive.
The interessed reader can found  examples of
diffeomorphisms that are not entropy expansive neither asymptotically
entropy expansive in \cite{Mi,PaVi}.

Next we recall the notion of dominated splitting.

\begin{Df}
\label{domi}
We say that a compact $f$-invariant set $\Lambda\subset M$ admits a
dominated splitting if the tangent bundle $T_{\Lambda}M$ has a
continuous $Df$-invariant splitting $E\oplus F$ and there exist
$C>0,\, 0<\lambda <1$, such that

\begin{equation} \label{domino}
\|Df^n|E(x)\|\cdot\|Df^{-n}|F(f^n(x))\|\leq C\lambda^n
\;\forall x\in \Lambda,\, n\geq 0.
\end{equation}
\end{Df}

Observe that if the topological entropy of a map $f:M\to M$
vanishes, $h(f)=0$, then automatically $f$ is $h$-expansive. For
instance Morse-Smale diffeomorphisms $\varphi:M\to M$ are
$h$-expansive. We remark that Morse-Smale diffeomorphisms are
$C^1$-stable under perturbations and so they constitute a class
which is robustly $h$-expansive.

 Here we are interested in
diffeomorphisms that exhibit a chaotic behavior, i.e.: their
topological entropy is positive. Moreover, we restrict our study to
homoclinic classes $H(p)$ associated to saddle-type hyperbolic periodic points.
Recall that the homoclinic class $H(p)$ of a saddle-type hyperbolic periodic
point $p$ of $f \in \dif^1(M)$ is the closure of the intersections
between the unstable manifold $W^ u(p)$ of $p$ and the
stable manifold  $W^s(p)$ of $p$.
These classes persist under perturbations and we wish to establish
the property of those classes under the assumption that
$h$-expansiveness is robust.

\begin{Df}\label{robuhagar}
Let $M$ be a compact boundaryless $C^\infty$ manifold and
 $f:M\to M$ be a $C^r$ diffeomorphism, $r\geq 1$. Let $H(p)$ be a
$f$-homoclinic class associated to the $f$-hyperbolic periodic point
$p$. Assume that there is a $C^r$ neighborhood ${\cal U}$ of $f$,
such that for any $g\in {\cal U}$ it holds that the continuation
$H(p_g)$ of $H(p)$  is $h$-expansive. Then we say that $f/H(p)$ is
$C^r$-robustly $h$-expansive.

\end{Df}

In \cite[Theorem B]{PaVi} we obtain that if $H(p,f)$ is isolated and
the finest dominated splitting on $H(p,f)$ is
$$T_{H(p,f)}M=E\oplus F_1\oplus\cdots\oplus F_{k}\oplus
G$$ with $E$ contracting, $G$ expanding and all $F_j$, $j=1,\ldots
,k$, one dimensional and not hyperbolic, then $f/H(p,f)$ is
$h$-expansive. Moreover, since the dominated splitting is preserved
under $C^1$-perturbations this result holds for a $C^1$-neighborhood
$\mathcal{U}(f)\subset\dif^1(M)$, i.e.: $h$-expansiveness is
$C^1$-robust.

Roughly speaking, \cite[Theorem B]{PaVi} says that the domination
property implies that small neighbourhoods in $H(p)$  have an
`ordered dynamics' and there cannot appear `arbitrarily small
horseshoes', i.e:, horseshoes generated by homoclinic points in
$W^s_{\xi}(x)\cap W^u_\xi(x)$ for $\xi>0$ arbitrarily small and
$x\in H(p)$ periodic, as in the example given in \cite{PaVi}[Section
2] for a surface diffeomorphism. The presence of these arbitrarily
small horseshoes would imply that $\sup_{x\in
H(p)}h(f,\Gamma_\epsilon(x))
> 0$ for any $\epsilon >0$.

This paper is intended to continue \cite{PaVi} in the reverse
direction: we analyze the consequences of $h$-expansiveness to hold
in a $C^1$-neighbourhood $\mathcal{U}(f)\subset \dif^1(M)$ of $f$.
Our main results are the following:

\begin{maintheorem} \label{prin1}
Let $M$, $f:M\to M$ and $H(p)$ be as in Definition \ref{robuhagar}
for $r=1$. Then $H(p)$ has a dominated splitting $E\oplus F$.
\end{maintheorem}

In fact  \cite[Example 2]{PaVi}  shows that in dimension
greater or equal to three  the existence of a dominated splitting
for $H(p)$ is not
enough tho guarantee $h$-expansiveness, so it is natural to search for
a stronger property.

Let us recall the concept of {\sl finest dominated splitting}
introduced in \cite{BDP}.
\begin{Df} \label{finest}
Let $\Lambda\subset M$ be a compact $f$-invariant subset such that
$TM/\Lambda=E_1\oplus E_2\oplus \cdots \oplus E_k$ with $E_j$ $Df$
invariant, $j=1,\ldots , k$. We say that $E_1\oplus E_2\oplus \cdots
\oplus E_k$ is dominated if for all $1 \leq j\leq k-1$
$$(E_1\oplus \cdots E_j)\, \oplus \, (E_{j+1}\oplus\cdots \oplus E_k)$$
has a dominated splitting. We say that $E_1\oplus E_2\oplus \cdots
\oplus E_k$ is the finest dominated splitting when for all
$j=1,\ldots ,k$ there is no possible decomposition of $E_j$ as two
invariant sub-bundles having domination.
 \end{Df}

An improvement of Theorem \ref{prin1} is the following.
\begin{maintheorem} \label{prin2}
Let $M$, $f:M\to M$ and $H(p)$ be as in Definition \ref{robuhagar}
for $r=1$. Then the finest dominated splitting in $H(p)$ has the
form $E\oplus F_1\oplus\cdots \oplus F_c\oplus G$ where all $F_j$
are one dimensional and not hyperbolic.
\end{maintheorem}

If $H(p)$ {\em is isolated} then we may refine the previous result.
Before we announce precisely this result, let us
recall the definitions of: chain recurrent set,  isolated
homoclinic class and heterodimensional cycles..
\begin{Df}
The chain recurrent set of a diffeomorphism $f$, denoted by $R(f)$,
is the set of points $x$ such that, for every $\epsilon > 0$, there
is a closed $\epsilon$-pseudo orbit joining $x$ to itself: there is
a finite sequence $x = x_0, x_1, \ldots , x_n = x$ such that
$\dist(f(x_i), x_{i+1}) < \epsilon$.
\end{Df}

\begin{Df}
We say that $H(p)$ is isolated if there are neighborhoods
$\mathcal{U}$ of $f$ in $\dif^1(M)$ and $U$ of the homoclinic  class
class $H(p)$ in $M$ such that, for every $g \in \mathcal{U}$, the
continuation $H(p_g)$ of $H(p)$ coincides with the intersection of
the chain recurrence set of $g$, $R(g)$ with  the neighborhood $U$.
\end{Df}

\begin{Obs}
Generically a recurrence class which contains a periodic point $p_g$
coincides with $H(p_g)$, {\rm \cite{BC}}.
\end{Obs}

\begin{Df}
 We say that $\Gamma$ is a cycle if $\Gamma=\{p_i, 0\leq i \leq n, p_0=p_n\}$,
where $p_i$ are hyperbolic periodic points of $f$ and $W^u(p_i)\cap W^s(p_{i+1})
\neq \emptyset$, for all $0\leq i \leq n-1$. $\Gamma$ is called a
heterodimensional cycle if, for some $i\neq j$, $\dim(W^u(p_i))\neq \dim(W^u(p_j))$.
\end{Df}
Recall that the {\em index} of a hyperbolic periodic point $p$ is the dimension of its
unstable manifold $W^u(p)$.

\begin{maintheorem} \label{prin3}
Let $M$, $f:M\to M$ and $H(p)$ be as in Definition \ref{robuhagar}
for $r=1$. Assume moreover that $f/H(p)$ is isolated. Then for $g$
in $\mathcal{U}(f)$, $H(p_g)$ has a dominated splitting of the form
$E\oplus F_1\oplus\cdots\oplus F_k\oplus G$ where $E$ is
contracting, $G$ is expanding and all $F_j$ are not hyperbolic and
$\dim(F_j)= 1$.
 Moreover, in case that the index of periodic points in $H(p_g)$ are in a $C^1$ robust way
 equal to $\mbox{index}(p)$ then for an open dense subset $\mathcal{V}\subset
 \mathcal{U}(f)$, $H(p_g)$ is hyperbolic, i.e.: $k=0$.
\end{maintheorem}

On the other hand, if there are $g$ arbitrarily $C^1$-close to $f$
such that in $H(p_g)$ there are periodic points of different index
then $H(p)$ is approximated by robust heterodimensional cycles,
\cite{BD1}.

\medbreak

If we do not assume that $H(p)$ is isolated but we know that $f$
cannot be approximated by $g$ exhibiting a heterodimensional cycle
we  have the following result:

\begin{maintheorem} \label{prin4}
Let $\Cy(M)=\{f \in \dif^1(M); f \mbox{  has no cycles}\}$, and
 $H(p)$ be as in Definition \ref{robuhagar}
for $r=1$. Assume that $f\in \dif^1(M)\backslash
\overline{\Cy(M)}$. Then for $g$ in a residual subset
$\mathcal{R}\subset \mathcal{U}(f)$, $H(p_g)$ has a dominated
splitting of the form $E^s\oplus E^c\oplus E^u$ where $E^c$ is not
hyperbolic and $\dim(E^c)\leq 2$, $E^s$ is contracting and $E^u$ is
expanding. Moreover, if $\dim(E^c)=2$ then $E^c=E^c_1\oplus E^c_2$
dominated.
\end{maintheorem}

\subsection{Idea of the proofs}

The proofs of Theorems \ref{prin1} and \ref{prin2} go by contradiction:
under the hypothesis that there is not a dominated splitting in $T_{H(p)}M$,
we profit from some ideas of \cite{PV} and \cite{Ro} to create a flat
tangency between $W^s(p)$ and $W^u(p)$.
We remark that in \cite{PV,Ro} for the case that $\dim (M)>2$ it was proved
that if $r\geq 2$ and
$g$ has a homoclinic tangency then there are diffeomorphisms
arbitrarily $C^r$-close to $g$ exhibiting persistent
homoclinic tangencies (thus generalizing results of \cite{Nh1}, see also \cite{Nh2}).
In our case, since we can perform the perturbations in the
$C^1$ topology, our arguments are simplier than theirs to
obtain a $C^2$ diffeomorphism $g$ exhibiting a flat tangency,
and afterward create an arc of tangencies between
$W^s(p)$ and $W^u(p)$.

Next we follow
\cite{DN},  to perform another $C^1$-perturbation with support in a small
neighborhood of the arc of tangencies leading to the
appearance of arbitrarily small horseshoes with positive entropy
contradicting $h$-expansiveness.
Therefore $Df/T_{H(p,f)}M$ admits a dominated spliting.

Moreover, either the finest dominated
splitting (see Definition \ref{finest}) has the form $E\oplus
F_1\oplus\cdots \oplus F_c\oplus G$ where all $F_j$ are one
dimensional and not hyperbolic or again we contradict robustness of
$h$-expansiveness using \cite[Theorem 6.6.8]{Go}.

For the proof of Theorem \ref{prin3} we assume some specific generic
properties described in Section \ref{genericos} and that $H(p)$ is
isolated.  These allow to prove that the extremal sub-bundles $E$
and $G$ are respectively contracting and expanding. Moreover if the
index of periodic points of $H(p_g)$ is robustly the index of $p$
then for an open dense subset of $\mathcal{U}(f)$ the dominated
splitting defined on $T_{H(p)}M$ is hyperbolic. This proof is done in two steps:
(1) First we prove in Lemma \ref{extremos1} that the extremal
sub-bundles are hyperbolic using the fact that $H(p)$ is isolated,
\cite{BDPR}. (2) Second  we show in Lemma \ref{indiceconstante} that if in
a $C^1$-robust way the index of periodic points in $H(p_g)$ are the
same for $g\in \mathcal{U}(f)$ then for an open and dense subset
$\mathcal{U}_1$ of $\mathcal{U}(f)$ we have that $H(p_g)$ is
hyperbolic.


Finally in Theorem \ref{prin4}, where we do not assume that $H(p)$
is isolated, we see, under the generic assumptions described at Section
\ref{genericos}, that for a residual subset $\mathcal{R}\subset
\mathcal{U}(f)$ we have a dominated splitting $E^s\oplus E^c\oplus
E^u$ defined on $T_{H(p)}M$ such that $E^s$ is contracting, $E^u$ is
expanding and $E^c$ is dominated and at most two dimensional. For
this we assume further that $f\in \dif^1(M)\backslash
\overline{\Cy}(M)$ which allows to use \cite[MainTheorem]{Cr}.

\section{Entropy expansiveness implies domination.}
In this section we prove Theorem \ref{prin2} assuming that $f/H(p)$
is robustly $h$-expansive.


Let $H(p)$ be a $f$-homoclinic class associated to the hyperbolic
periodic point $p$. Assume that there is a $C^1$ neighborhood ${\cal
U}$ of $f$ such that for any $g\in {\cal U}$ it holds that there is
a continuation $H(p_g)$ of $H(p)$ such that $H(p_g)$ is
$h$-expansive.

We may  assume that $p$ is a {\em hyperbolic fixed point} since $f/H(p)$ is
$h$-expansive if and only if $f^m/H(p)$ is $h$-expansive. This
follows from the fact that for any compact $f$-invariant set
$\Lambda$ we have that $h(f^m,\Lambda)=m \cdot h(f,\Lambda)$ which implies
that $h(f^m,\Lambda)=0\; \Longleftrightarrow \; h(f,\Lambda)=0$.

 Let $x\in W^s(p)\cap W^u(p)$ be a transverse homoclinic point
 associated to the periodic point $p\,$. We define $E(x)\equiv T_x W^s(p)$
 and $F(x)\equiv T_x W^u(p)$. Since $p$ is hyperbolic we have that
 $E(x)\oplus F(x)=T_x M$. Moreover, $E(x)$ and $F(x)$ are $Df$-invariant,
 i.e.: $Df(E(x))=E(f(x))$ and $Df(F(x))=F(f(x))$.
Denote by $H_t(p)$ the set of the transverse homoclinic points associated to $p$.
Then, it can be proved that $H(p)\equiv \overline{H_t(p)}$.
Here $\overline{A}$ stands for the closure in $M$ of the subset $A\subset M$.
 So if we prove that there is a dominated
 splitting for $H_t(p)$ we are done since we can extend by
 continuity the splitting to the closure $H(p)$.
 Moreover, since $C^2$-diffeomorphisms are dense in the $C^1$-neighbourhood $\mathcal{U}$
 we may assume that $f$ is of class $C^2$ taking into account that we are assuming that $h$-expansiveness
  is $C^1$-robust.

  We will use the following result proved in \cite{Fr}:
   \begin{Lem} \cite[Lemma 1.1]{Fr}\label{Fr}
 Let $M$ be a closed $n$-manifold,
 $f:M\to M$  a $C^1$ diffeomorphism, and  $\mathcal{U}(f)$  a given neighbourhood of $f$.
 Then, there
 exist $\mathcal{U}_0(f)\subset \mathcal{U}(f)$ and
 $\delta>0$ such that if $g\in \mathcal{U}_0(f)$,
  $S=\{p_1,p_2,\ldots p_m\}\subset M$ is a
 finite set, and $L_i,\,i=1,\ldots ,m$
 are linear maps, $L_i:TM_{p_i}\to TM_{f(p_i)}$, satisfying
 $\|L_i-D_{p_i}g\|\leq \delta,\, i=1,\ldots ,m$ then there is $\tilde{g}\in \mathcal{U}(f)$
  satisfying $\tilde{g}(p_i)=g(p_i)$ and $D_{p_i}\tilde{g}=L_i.$
  Moreover, if $U$ is any neighborhood of $S$
 then we may chose $\tilde{g}$ so that
 $\tilde{g}(x)=g(x)$ for all $x\in \{p_1,p_2\ldots p_m\}\cup (M\backslash
 U)$.
 \end{Lem}
 \begin{Obs}
 The statement given there is slightly different from that above,
 but the proof of our statement is contained in \cite{Fr}.
 \end{Obs}

\subsection{
Existence of dominated splitting: proof of Theorem \ref{prin1}.}
Under the hypothesis of Theorem \ref{prin1}, let us assume that $f$ is of class $C^r$, $r\geq 2$
 and prove that there is a dominated splitting for
 $H_t(p)$ .

The proof goes by contradiction and it is done in several steps:
(1) at Lemma \ref{le:dom} we perform a pertubation $g$ of $f$
exhibting a homoclinic point $x_g\in H(p_g)$ with small angle between
$W^s_{loc}(x_{g},g)$ and $W^u_{loc}(x_{g},g)$,
(2) at Proposition \ref{buenangulo}
we perform another perturbation (that we still denote by $g$)
 of $f$ to create a tangency
between $E^s(x, g)$ and $E^u(x,g)$, $x\in H(p_g)$,
(3)  at Proposition \ref{chato} through another pertubation of
$f$ we create an arc of flat tangencies $\beta \subset H(p_g)$,
(4) finally in Subsection \ref{herradurita} we perform a sequence
of perturbations of $f$ leading to $G$ near $f$ presenting a sequence of
two by two disjoint small horseshoes $H_{\epsilon_n} \subset H(p_G)$,
$\epsilon_n\to 0$ as $n\to \infty$.
Moreover, we can select the sequence $\epsilon_n$ in such a way that
none of then are a constant of $h$-expansiveness of $G$.
Since the entropy of each of these small horseshoes is positive,
we arrive to a contradiction to $h$-expansiveness of $f$.

To start, let us assume, by contradiction, that $H_t(p)$ has no
dominated splitting. Then,  by \cite[\S\ 3.6 Proof of Theorem F]{MPP} it holds
\begin{enumerate}
 \item[(AD)] for all $m\in\Z^+$ there exists $x_m$ such that
for all $0\leq n\leq m$,
$$\|Df^{n}|E(x_m)\|\cdot\|Df^{-n}|F(f^{n}(x_m))\|> 1/2 \; ,$$
\end{enumerate}


\begin{Lem}
\label{le:dom}
Assume that (AD) holds.
Then, given $\gamma>0$ and $\epsilon>0$ there is $m>0$ and $g$ an
$\epsilon$-$C^1$-perturbation of $f$ with a homoclinic point $x_{g}$
associated to $p_{g}$ such that the angle at $x_{g}$ between
$W^s_{loc}(x_{g},g)$ and $W^u_{loc}(x_{g},g)$ is less than $\gamma$.
\end{Lem}
\begin{proof}
Arguing by contradiction let us assume that there is $\gamma_0>0$
such that for all $g$ in ${\cal U}_0$  the angle at
$x_{g}$ between $W^s_{loc}(x_{g},g)$ and $W^u_{loc}(x_{g},g)$ is
greater or equal than $\gamma_0$.

By hypothesis there exist vectors $v_m\in F(x_m)$ and $w_m\in
E(x_m)$ with $\|v_m\|=\|w_m\|=1$ such that
$$\frac{\|Df^j(w_m)\|}{\|Df^j(v_m)\|}>\frac{1}{2}, \;\;\;\forall \,j,\, 1\leq j \leq m .$$
Take $\epsilon>0$ small such that any
$C^1$-$\epsilon$-perturbation of $f$  gives a
diffeomorphism $g\in{\cal U}_0$ where ${\cal U}_0$ is the
$C^1$-neighborhood of $f$ where we have $h$-expansiveness. Let
$\epsilon'>0$ be such that any perturbation of the derivatives along
a finite orbit of $f$ can be realized via Lemma \ref{Fr} by a
$C^1$-$\epsilon$-perturbation of $f$.

Let us define $T_j:T_{f^j(x_m)}M\to T_{f^j(x_m)}M$ a linear map
 such that $T_j|_{E(f^j(x_m))}=(1+\epsilon')id$ and
 $T_j|_{F(f^j(x_m))}=id$, $j=0,\ldots ,m$. Note that
 $T_j$ stretches $E=T_{x_m}W^s_\epsilon(x_m,f)$ and left $F=T_{x_m}W^u_\epsilon(x_m,f)$ unchanged.
 Let $P:T_{x_m}M\to T_{x_m}M$ be a linear map satisfying \,\,
 $P=id$ in $E(x_m)$ and $P=id+L$
 in $F(x_m)$ where $L:F(x_m)\to E(x_m)$ is a linear map such that
 $L(v_m)=\epsilon' w_m$ and $\|L\|=\epsilon'$.
 Finally define $G_0=T_1\cdot Df_{x_m},\cdot P$, and $G_j=T_{j+1}\cdot
 Df_{f^j(x_m)}$ for $j=1,\ldots ,m-1$.
 By Lemma \ref{Fr} there
 exists a diffeomorphism $g:M\to M$ such that
 $g$ is $\epsilon$-near $f$, keeps the orbit of $x_m$ unchanged for $j=0,1,\ldots
 ,m$, and such that $Dg_{f^j(x_m)}=G_j$.
 We may assume
 (and do) that the support of the perturbation does not cut a small
 neighborhood of $p$. It follows that $x_m$ continues to be a
 homoclinic point of $g$.
Moreover, we do not change $E(f^j(x_m))$, $j\in \Z$, and
$F(f^j(x_m))$ is changed only for $j\geq 0$. Thus such bundles  are the
stable and unstable directions of a homoclinic point of a
diffeomorphism $g\in {\cal U}_0$. We obtain that $v_m\mapsto
v_m+\epsilon' w_m=u$ and after $m$ iterates we have
$u_m=Dg^m(u)=Dg^m(v_m+\epsilon' w_m)=
Df^m(v_m)+(1+\epsilon')^mDf^m(\epsilon' w_m)$.

Given $\epsilon' >0$ we may find $m>0$ such that $\epsilon'
(1+\epsilon')^m\geq 4+2/\gamma_0$ where $\gamma_0>0$ is, by
hypothesis of absurd, such that $\angle(E(x),F(x))>\gamma_0$ for all
$x\in H_t(p_g)$, $g\in {\cal U}_0$, where $\angle(E(x),F(x))$
stands for the angle between $E(x)$ and $F(x)$. With this choice of
$m$, by \cite[Lemma II.10]{Ma2} we have
$$\|Df^m(v_m)\|=\|u_m -(1+\epsilon')^mDf^m(\epsilon' w_m)\|\geq $$
$$\geq \frac{\gamma_0}{1+\gamma_0}\|u_m\|\geq
\frac{\gamma_0}{1+\gamma_0}\,\big|\, \|\epsilon'
(1+\epsilon')^mDf^m(w_m)\|-\|Df^m(v_m)\|\,\big|\, .$$ Dividing the
inequality $\|Df^m(v_m)\| \geq  \frac{\gamma_0}{1+\gamma_0}\,\big|\,
\|\epsilon' (1+\epsilon')^mDf^m(w_m)\|-\|Df^m(v_m)\|\,\big| $ by
$\frac{\gamma_0}{1+\gamma_0}\|Df^m(v_m)\|$ and taking into account
that by hypothesis
$$\frac{\|Df^m(w_m)\|}{\|Df^m(v_m)\|}>\frac{1}{2}\,\,\mbox{ and }\,\, \epsilon'
(1+\epsilon')^m\geq 4+2/\gamma_0 $$ we find
$$\frac{1+\gamma_0}{\gamma_0} > \frac{\epsilon'
(1+\epsilon')^m}{2}-1 > 1+1/\gamma_0=\frac{1+\gamma_0}{\gamma_0}\, ,$$
arriving to a
contradiction. Hence $\angle(Dg^m(u),w_m) < \gamma$, proving Lemma \ref{le:dom}.

\end{proof}

 Let us recall the following result which may be found in \cite[Lemma 4.16]{BDP}, see also
\cite[Lemma 3.8]{BDPR}.

\begin{Teo} \label{BDP}
Let $p$ be a hyperbolic periodic point and $H(p)$ its homoclinic
class. Assume that $H(p)$ is not trivial. Then there exists and
arbitrarily small $C^1$-perturbation $g$ of $f$ and a hyperbolic
periodic point $q$ of $H(p_g)$ with period $\pi(q)$ and
homoclinically related with $p_g$ such that $Df_q^{\pi(q)}$ has only
positive real eigenvalues of multiplicity one.
\end{Teo}
Observe that in the previous result, since $q_g\in H(p_g)$, we have
$H(p_g)=H(q_g)$. So, to simplify notation, we may assume directly that
$p=q$ and moreover  that $g=f$,
and that $p$ is a fixed point.
We order the
eigenvalues of $Df_p$  labeling them as $0<\lambda_k<\cdots
<\lambda_1<1<\mu_1<\cdots <\mu_{d-k}$ so that the less contracting
and the less expanding ones are
respectively $\lambda_1$ and $\mu_1$. \\
By a small $C^1$-preturbations we may also assume that locally, in a
neighborhood $V$ of $p$, we have linearizing coordinates so that
$$f(x)=\sum_{j=1}^k \lambda_j a_j u_j \, +\, \sum_{j=1}^{d-k} \mu_j
a_{k+j}u_{k+j}$$ where we write  $x=\sum_{j=1}^d a_j u_j$ for $x\in V$	. \vspace*{2mm}

The lines in $W^s_{loc}(p)/V$ corresponding to the eigenvalues
$\lambda_j$ may be extended to all of $W^s(p)$ by backward iteration
by $f$ giving us a foliation by lines of dimension $k$. Similarly for $W^u(p)$ we
have a $(d-k)$-foliation by lines obtained by forward iteration by $f$.

Now, let us assume that $g$ is near $f$, $f=g$ in a
small neighborhood of $p$ and  that there is a small angle
between $T_x W^s(p,g)$ and $T_x W^u(p,g)$ where $x$ is a
$g$-homoclinic point associated to $p$. That is: there is $\gamma$ small such that
 $$\angle(T_x W^s(p,g),T_x W^u(p,g))<\gamma \, . $$
By Theorem \ref{BDP}, we may assume that all the eigenvalues
of $Df_{p}^{\pi_p}$ are
 positive with multiplicity one and that we have linearizing
 coordinates in a small neighborhood of $p$.


The next proposition stablishes that if the angle between
$T_x W^s(p,g)$ and $T_x W^u(p,g))$ is small than we can create a tangency
between $T_x W^s(p,\tilde{g})$ and $T_x W^u(p,\tilde{g}))$, for some $\tilde{g}$
near $g$.

\begin{Prop} \label{buenangulo}
There is $\gamma>0$ and $\mathcal{U}_0(g)\subset\mathcal{U}(f)$
so that for some $\tilde g\in\mathcal{U}_0(g)$ there is a tangency
between $E^s(x,\tilde g)$ and $E^u(x,\tilde g)$ if
$\angle(E^s(x,g),E^u(x,g))<\gamma$. Moreover $x$ is a homoclinic
point of $\tilde g$, $E^s(x,\tilde g)\oplus E^u(x,\tilde
g)$ has dimension $d-1$ and there is $N>0$ so that if $<u>$ is
the subspace common to $E^s(x,\tilde g)$ and $E^u(x,\tilde g)$ then
$(D\tilde{g})^N(<u>)$ is tangent to the line
corresponding to the less contracting eigenvalue and
$(D\tilde{g})^{-N}(<u>)$ is tangent to the line corresponding to the
less expanding eigenvalue of $D_p\tilde{g}$.
\end{Prop}

\begin{proof}
Let $\mathcal{U}(f)$, $\mathcal{U}_0(f)$ and $\delta$ be as in Lemma
\ref{Fr}. Shrinking $\mathcal{U}_0$ if it were necessary we may
assume that $\mbox{clos}\mathcal{U}_0(f)\subset \mathcal{U}(f)$.
Hence we may assume without loss of generality that there is some
$C>0$ such that $\sup\{\|D_xg\|:g\in\mathcal{U}_0(f)\}\le C.$

By hypothesis there is $g\in\mathcal{U}_0(f)$, $x\in
W^s(p_g,g)\trans W^u(p_g,g)$ and $\gamma>0$ small
so that
$$\angle(E^s(x,g),E^u(x,g)) < \gamma.$$

 Taking
$\gamma <\delta/C$ , since $\angle(E^s(x,g),E^u(x,g))<\gamma$, there
exist $v\in E^{s\bot}$ and $u\in E^s$ such that $v+u\in E^u,
\|u\|=1, \|v\|<\gamma.$ Let $T:T_x M\to T_x M$ be such that $T_{\big
|E^{s\bot}}=0,\,\,\, T(u)=-v$ and $\|T\|<\delta/C.$ Let
$L:T_{g^{-1}(x)}M\to T_xM$ be defined by $L=(Id+T)\circ
D_{g^{-1}(x)}g.$ Then we have
$$\|L-D_{g^{-1}(x)}g\|<\delta,\quad \mbox{and}\quad
u\in L(E^u(g^{-1}(x)).$$

 Take a neighborhood $U$ of $g^{-1}(x)$
such that $\mathcal{O}_g(x)\cap U=\{g^{-1}(x)\}$. Using Lemma
\ref{Fr} we find $\tilde{g}\in\mathcal{U}(f)$ such that
$g^j(x)=\tilde{g}^j(x)\;\mbox{for all }j\,, \;\tilde{g}=g$ outside
$U$, and $D_{g^{-1}(x)}\tilde{g}=L.$ Hence $x\in
W^s(p_{\tilde{g}},\tilde{g})\cap W^u(p_{\tilde{g}},\tilde{g})$ since
its forward and backward orbits continue to converge to
$p_{\tilde{g}}$. Moreover $u\in E^s(x,\tilde{g})\cap
E^u(x,\tilde{g})$ and so the intersection of $W^s(p_{\tilde{g}})$
and $W^u(p_{\tilde{g}})$ is not transverse at the point $x$.

Since the eigenvalues of $Df_p$ are all real positive and of
multiplicity one and $f=g$ in a small neighborhood of $p$, by $N$
forward iterations we have a vector $D^{N}\tilde{g}(u)$ almost
tangent to the straight line $<v_1>$ corresponding to the less
contracting eigenvalue at $p$. Again by Lemma \ref{Fr} we can
perturb $\tilde g$ outside a small neighborhood of $p$ to let the
direction of $(D\tilde{g})^{N}(u)$ coincide with $<v_1>$.
Similarly we obtain $(D\tilde{g})^{-N}(u)$
 tangent to the line corresponding to the less expanding eigenvector of
$D\tilde g_p$.
\end{proof}

From  Proposition \ref{buenangulo} we may assume for $f$ itself
that there is a homoclinic point of tangency $x\in W^s(p)\cap W^u(p)$
 with properties analogous to those of $\tilde g$.
The next lemma asserts that under these hypothesis, we can obtain
an arc $\beta$ of non-tranversal homoclinic points in $W^s(p)\cap W^u(p)$.

\begin{Prop}
 \label{chato}
 Let $p$ be a hyperbolic fixed point for $f$ of index $k$ and
 $x\in W^s(p)\cap W^u(p)$ such that the intersection at $x$ is not
 transversal. Then by an arbitrarily small $C^1$-perturbation we
 may obtain a diffeomorphism $g$ with $x\in W^s(p_g,g)\cap
 W^u(p_g,g)$ such that the intersection at $x$ is flat, there exists a small arc $\beta$
 contained in the intersection
 of the stable and unstable manifolds of $p$.
Moreover, there is $N>0$ such that $g^N(\beta)\subset W^s_{loc}(p,g)$ is
tangent to the eigenvector corresponding to the less contracting eigenvalue
and analogously $g^{-N}(\beta)\subset W^u_{loc}(p,g)$ is tangent to the
eigenvector corresponding to the less expanding eigenvalue.
 \end{Prop}
\begin{proof}
Since $p$ is a  hyperbolic saddle, $W^s(p)$ is an Euclidean $k$-dimensional hyperplane and
$W^u(p)$ an Euclidean $(d-k)$-dimensional hyperplane both immersed
in $M$.
 If the intersection at $x$ of $W^s(p)$ and
$W^u(p)$ is not transversal we should have a vector $u\neq 0$ in
$T_xW^u(p) \cap T_xW^s(p)$, i.e.: we have a tangency between
$W^s(p)$ and $W^u(p)$ at the homoclinic point $x$. Using Lemma
\ref{Fr} we may assume that the subspace generated by $u$ is the
unique in common between $T_xW^u(p)$ and $T_xW^s(p)$, that is
$T_xW^u(p) + T_xW^s(p)$ has dimension $d-1$. Moreover, we also may
assume that $k\geq d-k$ (otherwise we may take $f^{-1}$ instead of
$f$) and, again by Lemma \ref{Fr}, that  the tangent space
$T_xW^u_\epsilon(x)$ intersects trivially
$(T_xW^s_\epsilon(x))^\bot$ the orthogonal complement of
$T_xW^s_\epsilon(x)$. Under these assumptions the orthogonal
projection of $W^u_\epsilon(x)$ into $W^s_\epsilon(x)$ is locally a
diffeomorphism in a suitable neighborhood of $x$. Let us choose
$D_x\subset W^s_{\epsilon}(x)$ a small disk and $N>0$ such that
$f^N(D_x)\subset W^s_{\epsilon}(p)$, and let $L_x$ be a small disk
in $W^u_{\epsilon}(x)$ such that $f^{-N}(L_x)\subset
W^u_{\epsilon}(p)$. $L_x$ projects  onto $L'_x\subset D_x$
diffeomorphicaly.
 Via a local coordinate map we may identify $D_x$ with
$$\{y\in\R^d\,/\,y_{k+1}=\cdots =y_d=0;\, y_1^2+\cdots
+y_k^2=1\}\, ,$$ with $x$ identified with the origin $0$ and $u$
having the direction of $Oy_1$ which is tangent at 0 to $L'_x$ too.
$L_x$ may be viewed as the graph of a map $\Gamma:L'_x\to
(T_xW^s_\epsilon(x))^\bot$ with $\frac{\partial\Gamma}{\partial
y_1}\,|_{ 0}=0$. To simplify notation we write $(y_1,\ldots
,y_k)=Y_1$ and $(y_{k+1},\ldots ,y_d)=Y_2$. Hence if $(Y_1,Y_2)\in
L_x$ then $Y_2=\Gamma(Y_1(Z))$, where, given $L'_x$, $Y_1(Z)$ is a
local coordinate map from a
neighborhood of 0 in $\R^{d-k}$ to $D_x$.
\begin{claim}
\label{claim30}
There exists a  $C^1$ perturbation of $f$ that
produces a diffeomorphism $g\in {\cal U}(f)$ with a flat
intersection at $x\in D_x\cap L_x$, with $D_x \subset W^s_\epsilon(x)$ and
$L_x\subset W^u_\epsilon(x)$. This flat intersection contains a small arc $\beta$.
\end{claim}
\begin{proof}
Define $h:M\to M$ by
$$h(Y_1,Y_2)=(Y_1,Y_2-G(Y_1,Y_2)\Gamma(y_1,0\ldots ,0)).$$
Here $G$ is a $C^\infty$-bump function , $0\leq G(Y_1,Y_2)\leq 1$,
that vanishes in the boundary of the ball $B(0,\epsilon')$,
is equal to 1 in $B(0,\epsilon'/4)$, and such that $\|\nabla G\|< \frac{2}{\epsilon'}$, where $\nabla$ means the gradient.

Let us see that $h$ is a
diffeomorphism $\epsilon'$-$C^1$-close to the identity.
\begin{enumerate}
\item[(a)] $h$ is injective: Indeed,
$h(Y_1,Y_2)=h(Y'_1,Y'_2)$ implies that $Y_1=Y'_1$. Hence
$$Y_2-G(Y_1,Y_2)\Gamma(y_1,0\ldots,0)=
Y'_2-G(Y_1,Y'_2)\Gamma(y_1,0\ldots,0).$$ Therefore
$$\|Y_2-Y'_2\|=\|(G(Y_1,Y_2)-G(Y_1,Y'_2))\Gamma(y_1,0,\ldots,0)\|
\leq \|\Gamma(y_1,0,\ldots,0)\|,$$ where we have used that $0\leq
G(Z_1,Z_2)\leq 1$ for all $(Z_1,Z_2)$. Taking into account that
$$\Gamma(0,0)=0,\,\, \frac{\partial\Gamma}{\partial
y_1}\,\Big|_{ 0}=0 $$ we obtain that
$\Gamma(y_1,0\ldots,0)=o(\epsilon')$. Therefore
$$|(G(Y_1,Y_2)-G(Y_1,Y'_2))|=<\nabla G(Y_1,\Theta_2),Y_2-Y'_2>\leq
\|\nabla G\|\|\Gamma(y_1,0\ldots,0)\|<
\frac{2}{\epsilon'}o(\epsilon').$$ Here $(Y_1,\Theta_2)$ is a point
in the segment joining $(Y_1,Y_2)$ with $(Y_1,Y'_2)$.
 Let us choose $\epsilon'>0$ so small that
$\frac{2}{\epsilon'}\cdot o(\epsilon')<\frac{1}{2}$.
It follows that
$$\|Y_2-Y'_2\|=\|(G(Y_1,Y_2)-G(Y_1,Y'_2))\Gamma(y_1,0,\ldots,0)\|
\leq \frac{1}{2} \|\Gamma(y_1,0,\ldots,0)\|.$$ By induction we have
that for all $n\in\N$
$$\|Y_2-Y'_2\|=\|(G(Y_1,Y_2)-G(Y_1,Y'_2))\Gamma(y_1,0,\ldots,0)\|
\leq \frac{1}{2^n} \|\Gamma(y_1,0,\ldots,0)\|.$$
Therefore $Y_2=Y'_2$ and $h$ is injective.

\item[(b)]$h$ is a diffeomorphism:
 Indeed, we have
$$Dh=\left(\begin{array}{ccc}
 Id        & \vdots &   0    \\
 \hdots    & \vdots & \hdots \\
 -G\frac{\partial \Gamma}{\partial y_1}^t-
 \Gamma^t\frac{\partial G}{\partial Y_1} & \vdots    &
 Id-\Gamma^t\frac{\partial G}{\partial Y_2}      \\

\end{array}\right)
$$

Here $\Gamma=\Gamma(y_1,0\ldots,0)$, analogously $\frac{\partial
\Gamma}{\partial y_1}$  only  depends on $y_1$, and $\Gamma^t$ is the
transpose of $\Gamma$. As $\frac{\partial\Gamma}{\partial y_1}\,|_{
0}=0 $ we have that $-G\frac{\partial \Gamma}{\partial y_1}^t$ is
small if $\epsilon'$ is sufficiently small and the same is true with
respect to $\Gamma^t\frac{\partial G}{\partial Y_1}$ and
$\Gamma^t\frac{\partial G}{\partial Y_2}$, taking into account that
$\Gamma(y_1,0,\ldots,0)=o(\epsilon')$ and
$\|\nabla G\|<\frac{2}{\epsilon'}$.
Thus $Dh$ is invertible.
\end{enumerate}

 Items (a) and (b) above prove that $h$ is a diffeomorphism as $C^1$-close
to the identity map as we wish and $h=id$ off a small ball
$B(x,\epsilon')$.
Now consider $g=h\circ f$. Then $g$ is a small pertubation of $f$.
\begin{claim}
 \label{arcobeta}
$x$ is a flat $g$-homoclinic point and there is an arc
$\beta \subset W^s(p,g)\cap W^u(p,g)$ with $x \in \beta$.
\end{claim}

 Indeed, since $x\in W^s(p,f)\cap W^u(p,f)$ we have that $\lim_{n\to
+\infty}f^n(x)=\lim_{n\to -\infty}f^n(x)=p$ and so $x$ is neither
forward recurrent nor backward recurrent. This implies that we may
choose the support, $B(x,\epsilon')$, of the perturbation in such a
way that for $n\neq 0$, $g^n(B(x,\epsilon'))\cap
B(x,\epsilon')=\emptyset$.
Hence if $y\in W^s_\epsilon(x,f)$ then
for $\epsilon>0$ small we obtain that $y\in W^s_\epsilon(x,g)$.
But
$h$ sends and arc $\beta$ passing through $x$ in $W^u_\epsilon(x,f)$
onto an arc $\gamma$ included in
$W^s_\epsilon(x,f)=W^s_\epsilon(x,g)$ and passing through $x$ too.
Therefore $g^{-1}=f^{-1}\circ h^{-1}$ sends the arc $ \gamma$ into
$\beta$ which iterated sucessively by $f^{-1}$ converges to $p$.
Hence
$\beta$  is an arc contained in {\em both} the local stable and unstable
manifold of $x$ which is contained in $W^s(p,g)\cap W^u(p,g)$. Thus
$\beta$ is an arc of flat intersection between
 $W^s(p,g)$ and $W^u(p,g)$. This finishes both the proofs of Claim \ref{arcobeta}
and Claim \ref{claim30}.
\end{proof}

It is not difficult to see that this perturbation $g$ may be done in
such a way that for $N>0$ great enough $g^N(\beta)\subset
W^s_{loc}(p,g)$ is tangent to the eigenvector corresponding to the
less contracting eigenvalue and analogously $g^{-N}(\beta)\subset
W^u_{loc}(p,g)$ is tangent to the eigenvector corresponding to the
less expanding eigenvalue.

All together finishes the proof Proposition \ref{chato}.

\end{proof}


\subsubsection{Creating small horseshoes.} \label{herradurita}
The previous result gives  a diffeomorphism
$g$, $C^1$-near $f$,  such that the intersection between $W^u(p,g)$
and $W^s(p,g)$, in a local chart around $x$ such that $T_xW^s_\epsilon(x)\cap
T_xW^u_\epsilon(x)=<u>$, contains a segment
$\beta=\{su: -\delta\le s\le \delta\}$. Moreover,  $Dg^N u$ is tangent to
the line corresponding to the less contracting eigenvector of
$D g_p$ and $Dg^{-N} u$ is tangent to the
line corresponding to the less expanding eigenvector of $D g_p$.

Next we shall do a  perturbation of $g$, which will give a diffeomorphism $G$
such that $G$ coincides with $g$ outside a small neighborhood
of $\beta$, similar
to those of \cite[Lemma 5.1, Lemma 6.3]{DN} in order to create a sequence of
small horseshoes $H_n\subset H(p,G)$ associated to
$W^s_{loc}(x,G)$ and $W^u_{loc}(x,G)$.
These horseshoes will have positive topological entropy and will be built in such a way that
 neither $\epsilon>0$, nor $\epsilon/2,\,\epsilon/4,\,\ldots,\, \epsilon/2^n,\ldots $ will be constants of
$h$-expansiveness for $H(p,G)$.
Therefore the
diffeomorphism $G$ is not $h$-expansive, contradicting our hypothesis.

To do so we proceed as follows:
first, since we are working in a $C^1$-neighborhood of $f$ and $C^r,\, r\geq 2,$ diffeomorphisms are dense in $\dif^1(M)$ we may assume
that $g$, the diffeomorphism obtained at Proposition \ref{chato}, is of
class $C^r,\, r\geq 2$.
We split the proof into two cases, according to the index of $p$.

\subsection{$\mbox{index}(p)=d-1$}
Let us assume first that $p$ is of index $d-1$, i.e.:
$\dim(W^u(p,f))=1$. This will simplify the techniques involved. We
may assume, taking a large positive iterate by $g$ and possibly
reducing $\delta$, that $\beta$, the segment of tangency, is
contained in the local stable manifold of $p$ in a local chart which
is a linearizing neighborhood $U(p)$ of $p$.

 Let $\psi:[0,\delta]\to \R$ be a $C^\infty$ bump function
satisfying:
\begin{enumerate}
\item $\psi(s)=1/5$, for $s\in [0,\delta/16]$. This implies that
$\psi^{(k)}(0)=\psi^{(k)}(\delta/16)=0$ for all $k\geq 1$.
\item $\psi'(s)<0 $ for $s\in(\delta/16,\delta/8)$.
\item $\psi(s)=0$ for all $s\in [\delta/8,\delta/4]$, this implies
that  $\psi^{(k)}(\delta/8)=\psi^{(k)}(\delta/4)=0$ for all $k\geq 1$.
\item $\psi'(s)>0$ for $s\in (\delta/4,3\delta/8)$.
\item $\psi(s)=1$ for all $s\in [3\delta/8,\delta]$, this implies
that $\psi^{(k)}(3\delta/8)=\psi^{(k)}(\delta)=0$ for all $k\geq 1$.
\end{enumerate}


Next, consider $b:(-\delta,5\delta/4]\to\R$
such that
$$b(s)=\psi(s)\mbox{ for all } s\in [0,\delta]\, , $$
$$b(s)=\frac{1}{5}\psi(2(s+\delta/2))\mbox{ for all } s\in[-\delta/2,0]\, , $$
$$b(s)=\frac{1}{5^2}\psi(2^2(s+3\delta/4))\mbox{ for all }
s\in[-3\delta/4,-\delta/2]\, , $$
and in general
$$b(s)=\frac{1}{5^n}\psi(2^n(s+\delta(1-1/2^n))\mbox{ for all }
s\in[-\delta(1-1/2^n),-\delta(1-1/2^{n-1})]\, . $$
Put also
$$b(s)=5\psi(\frac{s-\delta}{2}) \mbox{ for } s\in[\delta,5\delta/4]\, .$$
It is easy to see that $b(s)$ is $C^\infty$ at $(-\delta,5\delta/4]$.
We may assume that for $s\in[0,\delta]$, $|b'(s)|\leq 24/\delta$ and
$|b''(s)|\leq K/\delta^2$, for some $K>0$. \\ Hence for $s\in [-\delta(1-1/2^n), -\delta(1-1/2^{n-1}]$
we have
$$|b'(s)|=\frac{1}{5^n}2^n\left|\psi'(2^n(s+\frac{2^{n}-1}{2^n}\delta))\right|\leq
\frac{24\cdot 2^n}{5^n\delta}$$
$\qquad $ and
$$ |b''(s)|=\frac{4^n}{5^n}\left|\psi''(2^n(s+\frac{2^{n}-1}{2^n}\delta))\right|\leq
\frac{4^nK}{5^n\delta^2}\, .$$
Therefore $|b'(s)|\to 0$ and $|b''(s)|\to 0$ when $s\to -\delta$.
Setting $b(-\delta)=0$ we have that $b'(-\delta)=b''(-\delta)=0$ and $b$ is of
class $C^2$ on $[-\delta,5\delta/4]$.

 Let $w$ be the unit vector in $T_xM$ tangent to the expanding eigenvector
of $Dg_p$.
Recall we are assuming that $\dim(W^u(p,G))=1$.
 Then $w$ is not contained in $T_xW^s(x,g)+T_xW^u(x,g)$ since
 $T_xW^u(x,g)$ is tangent to $T_xW^s(x,g)$.
Recall that $(0,s,0)$
are the coordinates of $\beta$ in a local chart
and that the interval
$(0,[-\delta, 5\delta/4],0)$ is totally contained in $\beta$.
In the plane given by the origin $0$
(identified with $x$) and the vectors $u$ and $w$ we consider the
graph of the function $\hat l:[\delta/4,5\delta/4]\to\R$ given by
$$\hat l(s)=\epsilon_1\cdot(s-\delta/2)(\delta-s)\, ,\quad
s\in [\delta/4,5\delta/4]\, .$$ Observe that for
$s\in[\delta/4,5\delta/4]$, $\hat l(s)$ vanishes at $s=\delta/2$
and $s=\delta$ and it has a maximum value equals to $\delta^2\epsilon_1/16$ at
$s=3\delta/4$.
 Now we extend $\hat l$ to $[-\delta,5\delta/4]$
in the following way:
$$\hat l(s)=\epsilon_2\cdot(s+\delta/4)(-s)\, ,\quad
s\in [-3\delta/8,\delta/8]\, ,$$
$$\hat l(s)=\epsilon_3\cdot(s+5\delta/8)(-\delta/2-s)\, ,\quad
s\in [-11\delta/16,-7\delta/16]\, ,$$
and in general for $n\geq 1$:
$$\hat l(s) = \epsilon_{n+1}\cdot (s+\delta(1-3/2^{n+1}))(-\delta(1-1/2^{n-1})-s)\, , \quad
s\in[-\delta(1-5/2^{n+2}),-\delta(1-9/2^{n+2})]\, .$$
For $s\in [-\delta(1-5/2^{n+2}),-\delta(1-9/2^{n+2})]$, $\hat l$
vanishes only at $s_{n_1}=-\delta(1-3/2^{n+1})$ and
$s_{n_2}=-\delta(1-1/2^{n-1})$ and it has a maximum value $\delta^2\epsilon_{n+1}/(5^n\cdot 2^{2n+4})$
at $(s_{n_1}+s_{n_2})/2$.
We complete the definition of $\hat l$ in $[-\delta,5\delta/4]$ setting $\hat l(s)=0$ elsewhere.

Finally, let $l(s)=\hat l(s)b(s)$ for all $s\in [-\delta,5\delta/4]$.
Then $l(s)$ is $C^\infty$ in $(-\delta,5\delta/4]$ and $C^2$ in $[-\delta,5\delta/4]$.

Put coordinates in the local chart $Y=(S,s,t)$ and denote by
$B_s$ a small $(d-1)$-dimensional disk around $x$ contained in a
fundamental domain of $W^s_{loc}(p,g)$ whose coordinates in the
local chart are $(S,s,0)$. Analogously
 denote by $B_u$  a small $1$-dimensional disk contained in $W^u(p,g)$ around $x$
 whose coordinates in the local chart are $(0,s,0)$.
Note that $B_s$ is characterized by
 $t=0$; and $B_u$ is the arc $\beta$ contained in $B_s$, parameterized
 by $s\in[-\delta,5\delta/4]$. The point $x$ is identified with $(0,0,0)$.

Now, pick another $C^\infty$ bump function $\varphi$ such
 that $\varphi$ vanishes outside a $\epsilon$ neighborhood of
 $\beta$, $\epsilon\geq 2\epsilon_1$, and is equal to 1 in the
 $\epsilon/2$ neighborhood of $\beta$.

Let $h:M\to M$ be given by
$$ \big(S,s,t\big)\mapsto \big(S,s,t+l(s)\varphi(\|Y\|)\big)$$
and $h=id$ outside $B(\beta,\epsilon)$ where $\epsilon$ is such that
the $\epsilon$-neighborhood of $\beta$ does not intersect $U\cap
g(U)\cap g^{-1}(U)$.

 Now, letting $G=h\circ g,$ we get,
by construction,  that $G$ is a small perturbation of $g$, and,
as in Proposition \ref{chato}, it is not difficult to see that
$B_s\subset W^s_{loc}(x, G)\subset W^s(p,G)$ and
$(0,s,l(s))\subset W^u_{loc}(x, G)\subset W^u(p,G).$
Furthermore, it is straightforward to show that $W^s(p,G)$ and
$W^u(p,G)$ intersect transversely at the points
$$(0,\delta/2,0),\, (0,\delta,0),\,(0,-\delta/4,0),\, (0,0,0),\ldots ,
(0,-\delta(1-3/2^{n+1}),0),\, (0,-\delta(1-1/2^{n-1}),0),\ldots   \,
$$ and the absolute value of the tangent of the angles at the points
$$(0,-\delta(1-3/2^{n+1}),0),\, (0,-\delta(1-1/2^{n-1}),0)\quad \mbox{ is
}\quad  \frac{\epsilon_{n+1}\delta}{5^n2^{n+1}},\; n\in\N \, . $$

We denote by $\beta'$ the graph of $l(s)$ in the plane $0uw$. If we choose
$\epsilon$,
$\epsilon_1\geq\epsilon_2\geq\cdots\geq\epsilon_n\geq\cdots$ with $\epsilon_n\searrow 0$ and
$\delta$ small, we may obtain the perturbation
$G=h\circ g$ to be $C^1$ small
(see \cite{Nh1}). Moreover, we can also assume that :
\begin{enumerate}
\item $G= g$ on $U\cap  g(U)\cap  g^{-1}(U)$, where we recall that $U=U(p)$ is a linearizing neighborhood of $p$.
\item $W^s_{loc}(p, g)=W^s_{loc}(p,G)$ and
$W^u_{loc}(p, g)=W^u_{loc}(p,G)$. Here $loc>0$ states for a suitable
small positive number,
\item $W^s_{loc}(x,G)\cup W^u_{loc}(x,G)\subset U\backslash\,
G(U)$. In particular $\beta\cup \beta'\subset U\backslash\, G(U)$.
\item $G^k(W^s_{loc}(x,G))\subset U$ for all $k\geq 0$ and there is $T>0$ such that
$G^{-k}(W^u_{loc}(x,G)) \subset U$ for all
$k\geq T$,
\item $ G^{-T}(\beta\cup\beta')\subset
U \backslash\, G^{-1}(U)$.
\end{enumerate}
We point out that item (5) above follows from
the fact that we may reduce the value of $\delta$, if it were
necessary, in order to ensure it.

\begin{Lem}
\label{inesperado}
There exists a sequence $\epsilon_n\searrow 0$ such that
$G$ is not $h$-expansive.
\end{Lem}
\begin{proof}
Recall that we are
working in a linearizing neighborhood $U$ of $p$ with respect to $g$.
Set
$$U^u_k=U\cap g(U)\cap\cdots\cap g^k(U)\quad \mbox{and}\quad U^s_k=U\cap g^{-1}(U)\cap
\cdots \cap g^{-k}(U)\, .$$

Let $\gamma'=G^{-T}(\beta')\subset U \backslash\,
G^{-1}(U)$ and denote by $(0,0,d_0),\, (0,0,d_\infty)$
the coordinates of the end points of $\gamma'$ corresponding
respectively to $s=5\delta/4$ and $s=-\delta$. In the same way we
label all points in $\gamma'$ corresponding to the \emph{transverse}
intersections of $\beta$ with $\beta'$: $(0,0,d_1)$ corresponds to
$(0,\delta/2,0)$ and $(0,0,d'_1)$ corresponds to $(0,\delta,0)$,
$(0,0,d_2)$ corresponds to $(0,-\delta/4,0)$ and $(0,0,d'_2)$
corresponds to $(0,0,0)$, $(0,0,d_3)$ corresponds to
$(0,-5\delta/8,0)$ and $(0,0,d'_3)$ corresponds to $(0,-\delta/2,0)$,
and so on, labeling
the image by $G^{-T}$ of all the
points of transverse intersection between $\beta$ and $\beta'$.

Take small arcs $a_1^s$ and $a_1'^s$ contained in $U\backslash
G^{-1}(U)$ tangent to the the direction of the
eigenvector corresponding to the weakest contracting eigenvalue of
$(DG)_p$ at the points $(0,0,d_1)$ and $(0,0,d'_1)$.
Multiply them by a $(d-2)$-dimensional disk $C$ of diameter $c$.
Analogously take small arcs $a_1^u$ and $a_1'^u$ tangent to the
direction corresponding to the eigenvector of the expanding
eigenvalue of $(DG)_p$
at the points $(0,\delta/2,0)$ and $(0,0,d'_1)$ and contained in
$U\backslash G(U)$.
By the $\lambda$-lemma,
 \cite{PdeM}[Lemma 7.1], the forward
orbits of $a_1^u$ and $a_1'^u$ contain arcs arbitrarily $C^1$ near
$W^u(p,G)$ and the backward orbits of $a_1^s\times C$ and
$a_1'^s\times C$ contain $(d-1)$-dimensional disks arbitrarily $C^1$
near $W^s(p,G)$. By the way we have chosen $a_1^s$ and
$a_1'^s$ and the assumption about the eigenvalues of
$D(G)_p$ (all positive real), we have that there is
 $k_1=k_1(\epsilon_1,\delta)$ such that for $k\geq k_1$ in $U$ we have
$\dist(G^{-k}(a_1^s),\beta)<\epsilon_1\delta^2/32$ and
$\dist(G^{-k}(a_1'^s),\beta)<\epsilon_1\delta^2/32$.
Moreover, we may choose $c>0$ small such that
$G^{-k}(a_1^s\times C)$ and
$G^{-k}(a_1'^s\times C)$ cut $\beta'$ but is contained in
the $\epsilon/4$ neighborhood of $\beta$ and therefore
$\varphi=1$ there.

 In the local coordinates we have chosen, we pick a thin rectangle $R_1$ with
 top and bottom given by $G^{-k_1}(a_1^s\times C)$ and
 $G^{-k_1}(a_1'^s\times C)$ and bounded in its sides by  segments parallel
 to the $w$-axis which is transverse to $D_S$.
Increasing $k_1$  and reducing $c$, $a_1^s$ and $a_1'^s$, if it were
necessary, we may assume that $G^{k_1}(R)$ is contained
in the $c$-neighborhood of the graph of $\beta'$ restricted to
$[3\delta/8,9\delta/8]$.

Set $g_1= G^{k_1}$ and let $g_2=G^T\big |
(U\backslash\, G^{-1}(U)): (U\backslash\,
G^{-1}(U))\to (U\backslash\, G(U))$ and consider
$$\Lambda_1=\bigcap_{n\in\Z} (g_2\circ g_1)^n(R_1)\, .$$
Then $\Lambda_1$ contains a horseshoe $H_1$ (see \cite{Nh1,DN}) and therefore
$H_{\epsilon_1}=\cup_{j=0}^{k_1+T-1} G(H_1)$ has positive
topological entropy. Since this horseshoe is arbitrarily small
we may assume that there is  a periodic point $p_1\in H_1$ such that
$H_1\subset \Gamma_\epsilon(p_1)$
see Definition \ref{gammadeepsilon}, where $0<2\epsilon_1\leq \epsilon$.
Moreover, the periodic point $p_1$ is homoclinically related to $p$
since by the $\lambda$-lemma we have that positive iterates by
$(g_2\circ g_1)^{-1}$ give thin subrectangles crossing all of $R_1$ and
hence the stable manifold of $p_1$ cuts $W^u_{loc}(x)\subset
W^u(p,G)$ and analogously positive iterates by $g_2\circ g_1$
gives subrectangles close to $\beta'$ in the Hausdorff metric and
therefore the unstable manifold of $p_1$ cuts $W^s_{loc}(x)\subset
W^s(p,G)$.

\begin{claim}
\label{paraentender}
There is  $\{\epsilon_n\}_{n=1}^\infty$ such
that for every $\epsilon_n$ it is associated a horseshoe $H_{\epsilon_n}$ with
$H_{\epsilon_n}\subset H(p,G)$ and $\lim_{n\to\infty}\diam(H_{\epsilon_n})= 0$.
\end{claim}
\begin{proof}
Let us choose $\epsilon_2>0$ and construct $H_{\epsilon_2}$. For this,  pick
$\epsilon_2\leq \epsilon_1$ such that
$G^{-k_1}(a_1^s\times C)$ and
$G^{-k_1}(a_1'^s\times C)$ are at a distance greater than
$\epsilon_2$ from $(S,s,0)$. Since $\epsilon_n\leq\epsilon_2$ for
all $n\geq 2$ we have that no part of the graph of $l(s)$ for $s\in
[-\delta,\delta/4]$ cuts $R_1$.

We found a new rectangle $R_2$ disjoint from $R_1$ contained in
$U^s_{k_2}\backslash\, U^s_{k_2+1}$ with $k_2> k_1$ applying again the $\lambda$-Lemma.
Increasing $k_2$  and reducing the corresponding values of $c_2$, $a_2^s$ and $a_2'^s$,
if it were necessary, we may assume that $G^{k_2}(R_2)$ is contained in the
$c_2$-neighborhood of the graph of $\beta'$ restricted to $[-5\delta/16,\delta/16]$.
By construction when we iterate by $G$ the images of $R_1$ and $R_2$ cannot intersect
since in $U\backslash\,G(U)$ there are only one iterate of $R_1$ and one iterate
of $R_2$ (namely $R_1$ and $R_2$).
We then have for $G$ two disjoint small horseshoes,
$H_1,\, H_2$  both with periodic points $p_1,p_2$ homoclinically related to $p$
(the proof that $p_2$ is homoclinically related to $p$ is the same than that to $p_1$).
Hence both $H_1$ and $ H_2$ are included in $H(p,G)$.

Next we choose $\epsilon_3\leq \epsilon_2\leq
\epsilon_1$ so that $G^{-k_2}(a_2^s\times C_2)$ and
$G^{-k_2}(a_2'^s\times C_2)$ are at a distance greater
than $\epsilon_3$ from $(S,s,0)$.
For such $\epsilon_3$, there
is a horseshoe $H_{\epsilon_3}$ disjoint from $H_{\epsilon_1}$ and
$H_{\epsilon_2}$ but still contained in $H(p,G)$. This construction
follows the same steps as before:
first find a thin rectangle $R_3$ cutting the
graph of $l(s)$ only for $s\in [-21\delta/32,-15\delta/32]$,
$R_3\cap R_1=\emptyset$, $R_3\cap R_2=\emptyset$. Then find an appropriate
positive real number
$k_3>k_2$ such that
 $G^{k_3}(R_3)$ is contained in the $c_3$-neighborhood of
the graph of $\beta'$ restricted to $[-21\delta/32,-15\delta/32]$.

In this way we may pick the sequence $\epsilon_n$ such that for
every $n$ it is associated
a horseshoe $H_{\epsilon_n}$ satisfying (1) $\lim_{n\to \infty}\diam(H_{\epsilon_n})\to 0$,
(2) $H_{\epsilon_j}\cap
H_{\epsilon_i}=\emptyset$ and (3) $H_{\epsilon_n}\subset
H(p,G)$ for all $n\in \Z^+$.  This proves Claim \ref{paraentender}.
\end{proof}

Since the topological
entropy of $H_{\epsilon_n}$ is positive for all $n$, and
$H_{\epsilon_n}\subset H(p,G)$, we conclude that
$G/H(p,G)$ is not $h$-expansive,
violating robustness of $h$-expansiveness. The proof of Lemma \ref{inesperado} is complete.
\end{proof}

%

Then, the final conclusion is that hypothesis (AD) described in the
begining of this section can not hold. In another words, we
 conclude that there
 exists $m>0$ such that for all homoclinic point $x\in H(p)$  there is
$1\leq k\leq m$ such that
$$\|Df^k/E(x)\|\, \|Df^{-k}/F(f^k(x))\| \leq \frac{1}{2}\, .$$

Following \cite[Theorem A]{SV},
it can be built  a dominated splitting for
 the homoclinic points of $H(p,f)$ as required, and then extend it by continuity to the whole $H(p,f)$ using that the closure of the homoclinic points
coincide with $H(p,f)$.

Thus in the case of $p$ a periodic point of index $d-1$ the proof of Theorem \ref{prin1} follows.

\begin{Obs}
Let us point out that even though we can assume that $g$, the
diffeomorphism with a segment of homoclinic tangencies, is
$C^\infty$, the bump function $l(s)$, used to perturb  it, is just
$C^2$. Hence it seems that a similar construction can be used to
prove the stronger result that $G/H(p)$ is not
asymptotically $h$-expansive. Recall, \cite{Bu, BFF}, that $C^\infty$- diffeomorphisms are asymptotically $h$-expansive so that a $C^\infty$ perturbation
 of a $C^\infty$ diffeomorphism does not disprove asymptotically $h$-expansiveness.
\end{Obs}
\subsubsection{$\mbox{index}(p)=k<d=\dim(M)$ }
For the general case of $\mbox{index}(p)=k<d=\dim(M)$ the proof is similar, the perturbation $h$ of $g$ given $G=h\circ g$ has to be adapted as we sketch below.

Let $w$ be the unit vector in $T_xM$ tangent to the less expanding eigenvector
of $Dg_p$.
 Then $w$ is not contained in $T_xW^s(x,g)+T_xW^u(x,g)$ see Propositions \ref{buenangulo} and \ref{chato}.
In a local chart around $x$, $(0,s,0)$ represent the coordinates of the arc $\beta$ but the coordinates $(S,s,T)$ are such that $S$ is a $(k-1)$-dimensional vector, and $T$ a $(d-k)$ dimensional vector that we split as $(t,T')=T$ with $t$ one-dimensional.    
As in the codimension one case we have that
$(0,[-\delta, 5\delta/4],0,0)$ is totally contained in $\beta$.
In the plane given by the origin $0$
(identified with $x$) and the vectors $u$ corresponding to $(0,1,0,0)$ and $w$ corresponding to $(0,0,1,0)$  we, as above, consider the
graph of the function $\hat l:[\delta/4,5\delta/4]\to\R$ given by
$$\hat l(s)=\epsilon_1\cdot(s-\delta/2)(\delta-s)\, ,\quad
s\in [\delta/4,5\delta/4]\, .$$ 
 Now we extend $\hat l$ to $[-\delta,5\delta/4]$ and define the $C^2$-function $l(s)$ 
as in the codimension one case.

Put coordinates in the local chart $Y=(S,s,t,T)$ and denote by
$B_s$ a small $k$-dimensional disk around $x$ contained in a
fundamental domain of $W^s_{loc}(p,g)$ whose coordinates in the
local chart are $(S,s,0,0)$. Analogously
 denote by $B_u$  a small $d-k$-dimensional disk contained in $W^u(p,g)$ around $x$
 whose coordinates in the local chart are $(0,s,0,T)$.
Note that $B_s$ is characterized by
 $t=0,\,T=0$; and $B_u$ contains the arc $\beta$ contained in $B_s$, parameterized
 by $s\in[-\delta,5\delta/4]$. The point $x$ is identified with $(0,0,0,0)$.

Now, pick a $C^\infty$ bump function $\varphi$ such
 that $\varphi$ vanishes outside a $\epsilon$ neighborhood of
 $\beta$, $\epsilon\geq 2\epsilon_1$, and is equal to 1 in the
 $\epsilon/2$ neighborhood of $\beta$.

Let $h:M\to M$ be given by
$$ \big(S,s,t,T\big)\mapsto \big(S,s,t+l(s)\varphi(\|Y\|),T\big)$$
and $h=id$ outside $B(\beta,\epsilon)$ where $\epsilon$ is such that
the $\epsilon$-neighborhood of $\beta$ does not intersect $U\cap
g(U)\cap g^{-1}(U)$.

 Now, letting $G=h\circ g,$ we get,
by construction,  that $G$ is a small perturbation of $g$, and,
as in Proposition \ref{chato}, it is not difficult to see that
$B_s\subset W^s_{loc}(x, G)\subset W^s(p,G)$ and
$(0,s,l(s),T)\subset W^u_{loc}(x, G)\subset W^u(p,G).$ 

The remaining of the proof of Theorem \ref{prin1} follows in a similar way 
to that of the codimension one case.

\section{Proof of Theorems \ref{prin2} and \ref{prin3}}
\label{genericos}

In this section we prove both Theorems \ref{prin2} and \ref{prin3}.
For this, let us first remark that
after \cite[\S 2.1]{ABCDW}, $C^1$-generically the finest dominated splitting
has a very special form. Thus, before we continue, let us first put
$f$ in that context.

\noindent {\em Generic assumptions.} There exists a residual subset
$\mathcal{G}$ of $\dif^1(M)$ such that if $f:M\to M$ is a
diffeomorphisms belonging to $\mathcal{G}$ then
\begin{enumerate}
\item
$f$ is Kupka-Smale, (i.e.: all periodic points are hyperbolic and
their stable and unstable manifolds intersect transversally)
\item for any pair of saddles $p,\, q$, either $H(p,f)=H(q,f)$ or $H(p,f)\cap H(q,f)=\emptyset$.
\item for any saddle $p$ of $f$, $H(p,f)$ depends continuously on $g\in\mathcal{G}$.
\item The periodic points of $f$ are dense in $\Omega(f)$.
\item The chain recurrent classes of $f$ form a partition of the chain recurrent set of $f$.
\item every chain recurrent class containing a periodic point $p$ is the homoclinic class
associated to that point.
\end{enumerate}

Taking into account \cite[Corollary, 6.6.2, Theorem 6.6.8]{Go},
that guarantees that the homoclinic tangency can be
associated to a saddle inside the homoclinic class,
the next result is proved in \cite[Corollary 3]{ABCDW}:

\begin{Teo}
\label{abcdw} (\cite[Corollary 3]{ABCDW}) There is a residual subset $\mathcal{I}\subset
\mathcal{G}$ of $\dif^1(M)$ such that if $f\in\mathcal{I}$ has a
homoclinic class $H(p,f)$ which contains hyperbolic saddles of
indices $i<j$ then either
\begin{enumerate}
\item For any neighborhood $U$ of $H(p,f)$ and any $C^1$-neighborhood
$\mathcal{U}$ of $f$ there is a diffeomorphism $g\in\mathcal{U}$
with a homoclinic tangency associated to a saddle of the homoclinic
class $H(p_g,g)$, where $p_g$ is the continuation of
$p$.\vspace*{3mm}
or
\item
There is a dominated splitting
$$T_{H(p,f)}M=
E \oplus F_1\oplus\cdots\oplus F_{j-i}\oplus G$$ with $\dim(E)=i$
and
 $\dim(F_h)=1$ for all $h$ and $\dim(G)=\dim(M)-j$. Moreover,
the sub-bundles $F_h$ are not hyperbolic.
\end{enumerate}
\end{Teo}

\noindent {\bf Proof of Theorem \ref{prin2}}.\/

Let $H(p)\subset M$ be a homoclinic class robustly entropy
expansive, i.e., there is a neighbourhood $\mathcal{U}\subset
\dif^1(M)$ such that $f\in\mathcal{U}$, there is a continuation
$H(p_g)$ of $H(p)$ for all $g\in\mathcal{U}$ and $H(p_g)$ is
$h$-expansive. By Theorem \ref{prin1} we have a dominated
splitting defined on $T_{H(p)}M$.  Moreover, by \cite[Theorem
6.6.8]{Go}, we have that in $H(p_g)$ there is a finest dominated
splitting which
 has the form
\begin{equation} \label{pregenerico}
T_{H(p_g,g)}M= E \oplus F_1\oplus\cdots\oplus F_{j-i}\oplus G
\end{equation}
 with
$E$, $G$ and $F_h$ $Df$-invariant sub-bundles, $h=1,\ldots , j-i$,
and all $F_h$ one-dimensional, and $$E\prec F_1\prec F_2\cdots \prec
F_{j-i}\prec G\, .$$ Otherwise, by the theorem of \cite{Go} cited
above, we may create with an arbitrarily small $C^1$-perturbation a
tangency {\em inside} the perturbed homoclinic class. After that we repeat the
arguments of \ref{herradurita} contradicting $h$-expansiveness.
 Theorem \ref{prin2} is proved.
 \medbreak

\noindent {\bf Proof of Theorem \ref{prin3}}.\/
 By \cite{CMP} there
is residual subset $\cR_0$ of $\mbox{Diff}^1(M)$ such that, for every
$f \in \cR_0$, any pair of homoclinic classes of f are either disjoint
or coincide. For $f \in \cR_0$, by \cite{Ab}, the number of different
homoclinic classes of $f$  is locally constant in $\cR_0$.
We split the proof into two cases: (1) this number is finite (and
in this case $f$ is {\em tame}) or (2) there are infinitly many
distinct homoclinic classes (and in this case $f$ is {\em wild}.
\medskip

\noindent {\bf $f$ is tame }\/
In this case $H(p)$ is isolated. Before we continue, recall
that if $V\subset M$ and  $\Lambda_f(V)$ is
 the maximal invariant set of $f$ in $V$, i.e.: $\Lambda_f(V)+\cap_{n\in\Z}f^n(V)$,
then set $\Lambda_f(V)$ is
robustly transitive if there is a $C^1$-neighbourhood $\mathcal{U}$
of $f$ such that $\Lambda_g(\overline{V})=\Lambda_g(V)$ and
$\Lambda_g(V)$ is transitive for all $g\in\mathcal{U}$ (i.e.:
$\Lambda_g(V)$ has a dense orbit).

\begin{Lem} \label{extremos1}
Assume $f:M\to M$ is tame and that
$T_{H(p)}M$ has a dominated splitting of the form
(\ref{pregenerico}). Then $E$ is contracting and $G$ is expanding.
\end{Lem}
\begin{proof}
Since $H(p)$ is isolated it is a robustly transitive set maximal
invariant in a neighbourhood $U\subset M$ and hence, according to
\cite{BDPR}[Theorem D],
 the extremal sub-bundles $E$ and $G$ are contracting and expanding respectively.
\end{proof}

Under the same hypothesis of the previous lemma either we have that
in a $C^1$-robust way the index of periodic points in $H(p_g)$, $g$
near $f$, are the same and equal to $\mbox{index}(p)$ or there are
$g$ arbitrarily $C^1$-close to $f$ such that in $H(p_g)$ there are
periodic points of different index.
In the first case we have
\begin{Lem} \label{indiceconstante}
There is a dense open subset $\mathcal{U}_1$ of $\mathcal{U}(f)$
in the $C^1$ topology such that
for all $g\in \mathcal{U}_1$ we have that $H(p_g)$ is hyperbolic.
\end{Lem}
\begin{proof}
We follow the lines of the proof at \cite[Section 6]{BD1}. Since
$H(p)$ is isolated by \cite[Corollaire 1.13]{BC} or \cite[Theorem
A]{Ab} it is robustly isolated. Let $E$ and $F$ be sub bundles such
that $T_{H(p_g)}M=E\oplus F$ is $m$-dominated, for all $g\in
\mathcal{U}(f)$, with $\dim(E)=\mbox{index}(p)$. We need to prove
that $\|Df^n_{/E(x)}\| \to 0$ as $n\to +\infty$ and
$\|Df^{-n}_{/F(x)}\| \to 0$ as $n\to +\infty$ for any $x\in H(p_g)$
in order to prove that $H(p_g)$ is hyperbolic. Let us show only that
$\|Df^n_{/E(x)}\| \to 0$ as $n\to +\infty$, the other one being
similar. For this, it is enough to show that for any $x\in H(p_g)$
there exists $k=k(x)$ such that
$\prod_{i=0}^k\|Dg^m_{/E(g^{im}(x))}\|<\frac{1}{2}.$

  Arguing by contradiction, assume this does not hold. Then, there
exist  $z\in H(p_g)$ such that
$\prod_{i=0}^k\|Df^m_{/E(f^{im}(z))}\|\ge \frac{1}{2}\;\;\forall k\ge 0.$

 As in the proof of \cite[Theorem B]{Ma2} we may find
  $y\in H(p_g)\cap \Sigma(g)$, where $\Sigma(g)$ is a set of total probability measure, such that
  $$\lim_{n\to +\infty}\frac{1}{n}\sum_{i=0}^{n-1}\log\|D_{g^{mi}(y)}g^m\,
\big/\,E(g^{mi}(y)\|\geq 0$$

  Thus there is a perturbation $h$ of $g$ such that $h$ has a non hyperbolic periodic point in $H(p_h)$.
  After a new perturbation we obtain periodic points $P$ and $Q$
   contained in a small neighborhood $U$ of $H(p_h)$ and with different indeces.
   Since $H(p)$ is $C^1$-robustly isolated $P,Q\in H(p_h)$ contradicting our assumption
  that in a $C^1$-robust way the index of periodic points in $H(p)$ are the same
and equal to $\mbox{index}(p)$. This finishes the proof of Theorem \ref{prin3} in this case.
\end{proof}

\medskip

In the second case, that is,
there are $g$ arbitrarily $C^1$-close to $f$ such that in $H(p_g)$ there are
periodic points of different indeces,
 by \cite{GW}, $C^1$-generically
the diffeomorphism $g$, and hence $f$, can be $C^1$ approximated by
diffeomorphisms exhibiting a heterodimensional cycle.  Next we show
 that in this case the eigenvalues of periodic points are robustly
in $\R$.

\begin{Lem} \label{extremos2}
Let us assume that there is a periodic point $q\in H(p)$ with
expanding complex eigenvalues such that
$\mbox{index }(q)<\mbox{index }(p)$. Then there is an arbitrarily
$C^1$-small perturbation of $f$ creating a tangency inside the
perturbed homoclinic class $H(p_g)$.
\end{Lem}
\begin{proof}
$C^1$ generically we may assume that  there is a robust heterodimensional cycle
between $p$ and $q$ and that $W^s(p)\cap W^u(q)$ contains a compact arc $l$
homeomorphic to $[0,1]$, (see \cite{BD1}).
Let us consider a disk $D$ of the same dimension $s$ of $W^s(p)$ and
contained in $W^s(p)$ such that
 $D$ is homeomorphic to $[0,1]\times [-1,1]^{s-1}$ by a homeomorphism $h$ such that
 $h([0,1]\times \{0\}^{s-1}=l$.
 Iterating by $f^{-\pi(q)}$ this arc $l$ spiralizes around $q$ while
$D$ stretches approaching $W^s(q)$.
Since $W^s(q)\cap W^u(p)\neq\emptyset$ there is a $C^1$ small
perturbation of $f$ creating a tangency between $W^s(p_g)$ and $W^u(p_g)$.
\end{proof}

\begin{Cor} \label{acima}
If there is a periodic point $q\in H(p)$  with expanding complex eigenvalues
such that $\mbox{index }(q)<\mbox{index }(p)$ then $H(p)$ is not $C^1$ robustly $h$-expansive.
\end{Cor}
\begin{proof}
Under the hypothesis of the lemma we may create tangencies inside
$H(p)$ and by another
$C^1$- perturbation an arbitrarily small horseshoe in the
intersection between $W^s_{loc}(p)$ and $W^u_{loc}(p)$ contradicting
$h$-expansiveness.
\end{proof}

Thus Corollary \ref{acima} implies that  the eigenvalues of periodic
points in $H(p)$ are real numbers in a robust way. By \cite{ABCDW}
for $C^1$-generic diffeomorphisms the set of indices of the
(hyperbolic) periodic points in a homoclinic class form an interval
in $\N$. Thus by \cite{BD1}[Theorem 2.1] there are diffeomorphisms
arbitrarily $C^1$-close to $f$ with $C^1$-robust heterodimensional
cycles.

As a consequence we obtain in both cases the following result:
 \begin{Teo} \label{solita}
If $f/H(p)$ is $C^1$ robustly $h$-expansive and $H(p)$ is an
isolated homoclinic class then for a dense open subset
$\mathcal{U}'\subset \mathcal{U}(f)$ either $f/H(p)$ is hyperbolic
and we have $T_{H(p)}M=E^{s}\oplus E^{u}$ or there is a robust
heterodimensional cycle in $H(p_g)$ for $g$ arbitrarily close to
$f$.
 \end{Teo}
\begin{proof}
If we have that in a $C^1$-robust way the index of periodic points
in $H(p_g)$ are the same and equal to $\mbox{index}(p_g)$ by Lemma
\ref{indiceconstante} there is an open dense subset $\mathcal{V}$ of
$\mathcal{U}(g)$ such that $H(p_g)$ is hyperbolic for $g\in
\mathcal{V}$. Hence we are done. Otherwise we have an open subset
$\mathcal{U}(g)$ in any neighborhood $\mathcal{V}\subset
\mathcal{U}(f)$ of any diffeomorphism $g\in\mathcal{U}(f)$
exhibiting a heterodimensional cycle, \cite{BD1}.
 This finishes the proof Theorem \ref{solita}, which in its turn
gives the proof of Theorem \ref{prin3}.
\end{proof}

\medbreak
\noindent {\bf $f$ is wild}\/
Now let us assume that $H(p)$ is not isolated. Either there is a
small $C^1$-perturbation $g$ of $f$ such that $H(p_g)$ is isolated
 or $H(p)$ is persistently not isolated, i.e.: $H(p_g)$ is not isolated
for any $g$ close to $f$. In the first case  we are done by Theorem
\ref{solita}.

 In the second case the
following result of \cite{Cr} (see also \cite{W}) is valid assuming that $f$ is far from
homoclinic cycles.
\begin{Obs}
Since $f/H(p)$ is $h$-expansive we are far from homoclinic
tangencies.
\end{Obs}

\begin{Teo}[Crovisier] \label{Crovisier}
 There exists a dense $G_\delta$ subset of $\dif^1(M)\backslash
\overline{\rm{Tang}\cup\rm{Cycles}}$ such that each homoclinic class
$H$ has a dominated splitting $T_HM = E^s \oplus E^c_1 \oplus E^c_2
\oplus E^u$ which is partially hyperbolic and such that each central
bundle $E^c_1,E^c_2$ has dimension 0 or 1.
\end{Teo}

 Thus  Theorem \ref{prin4} is a consequence of Theorem
 \ref{Crovisier} and the previous remark.

\begin{tabbing}
Universidade Federal do Rio de Janeiro, \hspace{1cm}\= Facultad de Ingenieria, \kill
 M. J. Pacifico, \> J. L. Vieitez, \\
Instituto de Matematica, \> Instituto de Matematica,\\
Universidade Federal do Rio de Janeiro, \> Facultad de Ingenieria,\\
C. P. 68.530, CEP 21.945-970, \> Universidad de la Republica,\\
Rio de Janeiro, R. J. , Brazil. \> CC30, CP 11300,\\
                                \> Montevideo, Uruguay\\
{\it pacifico@im.ufrj.br} \> {\it jvieitez@fing.edu.uy}
\end{tabbing}

\end{document}